\documentclass[a4paper, 11pt]{article}
\usepackage[round]{natbib}
\newcommand{\fullname}{Adam D.\ Bull}
\newcommand{\thetitle}{A Smirnov-Bickel-Rosenblatt theorem for 
compactly-supported wavelets}

\newcommand{\mscone}{62G20}
\newcommand{\msctwo}{62G07}
\newcommand{\mscthree}{62G08}
\newcommand{\mscfour}{62G15}
\newcommand{\mscfive}{65T60}
\newcommand{\themsclass}{\mscone\ (Primary); \msctwo, \mscthree, \mscfour, 
\mscfive\ (Secondary)}

\newcommand{\kwdone}{nonparametric statistics}
\newcommand{\kwdtwo}{compactly-supported wavelets}
\newcommand{\kwdthree}{asymptotic distribution}
\newcommand{\kwdfour}{confidence sets}
\newcommand{\kwdfive}{supremum norm}
\newcommand{\thekeywords}{\kwdone, \kwdtwo, \kwdthree, \kwdfour, \kwdfive}

\newcommand{\addressone}{Statistical Laboratory}
\newcommand{\addresstwo}{University of Cambridge}

\newcommand{\theemail}{a.bull@statslab.cam.ac.uk}

\newcommand{\theabstract}{
In nonparametric statistical problems, we wish to find an estimator of an 
unknown function \(f.\) We can split its error into bias and variance terms; 
Smirnov, Bickel and Rosenblatt have shown that, for a histogram or kernel 
estimate, the supremum norm of the variance term is asymptotically distributed 
as a Gumbel random variable. In the following, we prove a version of this 
result for estimators using compactly-supported wavelets, a popular tool in 
nonparametric statistics.
Our result relies on an assumption on the nature of the wavelet, which must be 
verified by provably-good numerical approximations. We verify our assumption 
for Daubechies wavelets and symlets, with \(N = 6, \dots, 20\) vanishing 
moments; larger values of \(N,\) and other wavelet bases, are easily checked, 
and we conjecture that our assumption holds also in those cases.
}

\usepackage{amsmath, amsfonts, amssymb, amsthm, mathtools, thmtools,
booktabs, tikz, mparhack, verbatim, algorithm, algorithmic} 
\usepackage[bookmarks, breaklinks, colorlinks, linkcolor=blue, citecolor=blue, 
pdfauthor={\fullname}, pdftitle={\thetitle}]{hyperref}
\let\oldmarginpar\marginpar
\renewcommand\marginpar[1]{\-\oldmarginpar[\raggedleft\footnotesize 
#1]{\raggedright\footnotesize #1}}

\DeclarePairedDelimiter{\abs}{\lvert}{\rvert}
\DeclarePairedDelimiter{\norm}{\lVert}{\rVert}
\newcommand{\N}{\mathbb{N}}
\newcommand{\Z}{\mathbb{Z}}

\newcommand{\R}{\mathbb{R}}
\renewcommand{\P}{\mathbb{P}}
\newcommand{\E}{\mathbb{E}}
\newcommand{\Var}{\mathbb{V}\mathrm{ar}}
\newcommand{\Cov}{\mathbb{C}\mathrm{ov}}
\newcommand{\iid}{\overset{\mathrm{i.i.d.}}{\sim}}

\numberwithin{equation}{section}
\declaretheorem[numberwithin=section]{theorem}
\declaretheorem[sibling=theorem]{lemma}

\declaretheorem[sibling=theorem]{assumption}

\newcommand{\fh}{\hat f}
\newcommand{\fb}{\bar f}
\newcommand{\sfb}{\overline \sigma_\varphi}

\begin{document}

\title{\thetitle
\footnotetext{\emph{Mathematics subject classification 2010.} \themsclass}
}
\author{\fullname\\\footnotesize \addressone\\\footnotesize \addresstwo\\
\footnotesize \theemail}
\date{}

\maketitle

\begin{abstract}
  \theabstract
\end{abstract}

\section{Introduction}
\label{sec:introduction}

In nonparametric statistical problems, such as density estimation, regression, 
or white noise, we wish to find an estimate \(\hat f\) of an unknown function 
\(f\) \citep{tsybakov_introduction_2009}. We can measure the accuracy of an 
estimator \(\hat f\) by its distance from \(f,\)
\(\norm{\hat f - f},\)
where \(\norm{\,\cdot\,}\) is some norm on functions. We can then decompose 
the error into variance and bias terms,
\[\norm{\hat f - f} \le \norm{\hat f - \E \hat f} + \norm{\E \hat f - f},\]
where the bias term \(\norm{\E \hat f - f}\) is deterministic, and the 
variance term \(\norm{\hat f - \E \hat f}\) we hope has an asymptotic 
distribution independent of \(f.\)

In density estimation, for the supremum norm on \([0,1],\)
\[\norm{f}_\infty \coloneqq \sup_{x \in [0, 1]} \abs{f(x)},\]
the limiting distribution of a suitably scaled variance term is given by 
\citet{smirnov_construction_1950} for histograms, and in the classical paper 
of \citet{bickel_global_1973} for kernel estimates. In both cases, as the 
sample size \(n\) tends to infinity, the variance term approaches a Gumbel 
distribution,
\[\P\left(A_n\left(\norm*{\frac{\hat f_n - \E \hat f_n}{\sqrt{f}}}_\infty - 
B_n\right) \le x\right) \to e^{-e^{-x}},\]
for known sequences \(A_n,\) \(B_n.\) This result has been of key importance 
for a variety of problems in nonparametric statistics.

Wavelets are an increasingly popular statistical tool, allowing a simple 
theoretical description of nonparametric problems, and a computationally 
efficient implementation of their solution.  \citet{gine_confidence_2010} 
establish an equivalent of these Smirnov-Bickel-Rosenblatt theorems for 
certain wavelet estimators, using a result of 
\citet*{husler_convergence_2003} on the convergence of cyclostationary 
Gaussian processes.  \citeauthor{gine_confidence_2010} describe the asymptotic 
distribution of the supremum, on increasing intervals, of the Gaussian process
\[X(x) \coloneqq \int K(x, t)\,dB_t,\]
where \(K\) is a wavelet projection kernel, and \(B\) a Brownian motion;
they then link this result to the statistical problem considered above.

Their result holds only for wavelets satisfying certain analytic conditions, 
which the authors demonstrate are satisfied by Battle-Lemari\'{e} wavelets 
having \(N \le 4\) vanishing moments; \citet*{gine_periodized_2011} extend 
this to larger values of \(N.\) Past work has not, however, succeeded in 
establishing results for the most commonly used wavelets, such as Daubechies 
wavelets and symlets. These wavelets, unlike those of Battle and Lemari\'{e}, 
are compactly supported, allowing the most efficient implementation of 
statistical procedures. In the following, we demonstrate that the conditions 
of \citet{gine_confidence_2010} hold also in these cases, thereby proving a 
Smirnov-Bickel-Rosenblatt theorem for the most practically relevant wavelet 
bases.

We work primarily in the white noise model, but also discuss consequences for 
the density estimation and regression models. We consider wavelet bases both 
on \(\R,\) and also on the interval, using the construction of 
\citet{cohen_wavelets_1993}.  In both cases, we show that the variance term 
again approaches a Gumbel distribution. We also extend a theorem of 
\citet{husler_convergence_2003} (as reported in 
\citealp{husler_extremes_1999}), establishing a uniform convergene result for 
cyclostationary processes; this allows us to show that convergence to Gumbel 
occurs uniformly in large values of the level \(x.\) These results are used in 
\citet{bull_honest_2011} to construct adaptive confidence bands for 
nonparametric statistical problems, and are also of relevance to many other 
wavelet procedures.

To prove our results, we must first verify an assumption on the wavelet 
functions, which in general do not have an analytic form. We therefore make 
use of provably-accurate numerical approximations, given by 
\citet{rioul_simple_1992}; these approximations also provide an efficient means 
of computing the constants in our results.  We verify our assumption for 
Daubechies wavelets and symlets, having \(N= 6, \dots, 20\) vanishing moments 
\citep[\S6.4]{daubechies_ten_1992}; however, the numerical approximations can 
easily be applied to larger values of \(N,\) and other wavelet bases, and we 
conjecture that our assumption holds also in these cases.

We state our result in \autoref{sec:results}, and describe the necessary 
numerical approximations in \autoref{sec:numerics}. We give proofs in 
\autoref{app:proofs}, and source code in \autoref{app:source}.

\section{Results}
\label{sec:results}

To begin, we will need \(\varphi\) and \(\psi,\) the scaling function and 
wavelet of an orthonormal multiresolution analysis on \(L^2(\R).\) (For an 
introduction to wavelets and their statistical applications, see 
\citealp{hardle_wavelets_1998}.) We make the following assumptions on 
\(\varphi\) and \(\psi,\) which are satisfied, for example, by Daubechies 
wavelets and symlets, with \(N \ge 6\) vanishing moments 
(\citealp[\S6.1]{daubechies_ten_1992}; \citealp[\S14]{rioul_simple_1992}).

\begin{assumption}\ \label{ass:wavelet-basis}
\begin{enumerate}
\item For \(K \in \N,\) \(\varphi\) and \(\psi\) are supported on the interval 
  \([1-K, K].\)
\item For \(N \in \N,\) \(\psi\) has \(N\) vanishing moments:
\[\int_\R x^i \psi(x) \, dx = 0, \qquad i=0,\dots,N-1.\]
\item \(\varphi\) is twice continuously differentiable.
\end{enumerate}
\end{assumption}

We will consider wavelet bases on both \(\R\) and \([0, 1],\) constructed from 
\(\varphi\) and \(\psi.\) On \(\R,\) we have an orthonormal basis of 
\(L^2(\R)\) given by
\begin{equation}
\label{eq:wavelets}
\varphi_{j_0, k}(x) \coloneqq 2^{j_0/2}\varphi(2^{j_0}x - k), \qquad \psi_{j, 
k} \coloneqq 2^{j/2}\psi(2^jx-k),
\end{equation}
for some lower resolution level \(j_0 \in \Z,\) \(j > j_0,\) and \(k \in \Z.\)

On \([0, 1],\) we can generate an orthonormal basis of \(L^2([0, 1])\) using 
the construction of \citet{cohen_wavelets_1993} (see also 
\citealp{chyzak_construction_2001}). We obtain basis functions 
\[\varphi_{j_0,k}, \qquad k = 0, \dots, 2^{j_0}-1,\]
and
\[\psi_{j,k}, \qquad j > j_0, \, k = 0, \dots, 2^j-1.\]
For \(k \in [N, 2^j - N),\) these functions are given by \eqref{eq:wavelets}. 
For other values of \(k,\) the basis functions are specially constructed, so 
as to form an orthonormal basis with desired smoothness properties.

We will also need to make an assumption on the precise form of the scaling 
function \(\varphi.\) While this assumption is difficult to verify 
analytically, we will see in the following section it can be tested using 
provably good numerical approximations.

\begin{assumption}
  \label{ass:sigma-maximum} The 1-periodic function
  \[\sigma^2_\varphi(t) \coloneqq \sum_{k \in \Z} \varphi(t - k)^2\]
  attains its maximum \(\overline \sigma^2_\varphi\) at a unique point \(t_0 
  \in [0, 1),\) and \((\sigma^2_\varphi)''(t_0) < 0.\)
\end{assumption}

Given these assumptions, suppose we have an unknown function \(f,\) with 
empirical wavelet coefficients \(\alpha_k, \beta_{j,k},\)
\[f \coloneqq \sum_k \alpha_k \varphi_{j_0, k} + \sum_{j > j_0}\sum_k 
\beta_{j,k}\psi_{j,k}.\]
Suppose also that we observe the empirical wavelet coefficients
\begin{equation}
\label{eq:empirical}
\hat \alpha_k \coloneqq \alpha_k + \epsilon_{j_0, k}, \qquad \hat \beta_{j, k} 
\coloneqq \beta_{j, k} + \epsilon_{j, k},
\end{equation}
where the \(\epsilon_{j, k}\) are i.i.d.\ \(N(0, \sigma^2).\) This is the case 
in the white noise model, where we observe the process
\[Y_t = \int_0^t f(s)\,ds + n^{-1/2} B_t,\]
for a Brownian motion \(B.\) The empirical wavelet coefficients
\[\hat \alpha_k = \int \varphi_{j_0, k}(t)\, dY_t, \qquad \hat \beta_{j, k} = 
\int \psi_{j, k}(t)\, dY_t,\]
satisfy \eqref{eq:empirical} with \(\sigma^2 = n^{-1}\) 
\citep[\S10]{hardle_wavelets_1998}. The model \eqref{eq:empirical} also 
serves as a limiting approximation in density estimation and regression, which 
we return to later.

The wavelet projection estimate of \(f,\) at resolution level \(j,\) is then
\[\hat f(j) \coloneqq \sum_{k} \hat \alpha_k \varphi_{j_0, k} + \sum_{j_0 < l 
\le j} \sum_k \hat \beta_{l, k} \psi_{l, k}.\]
Set
\begin{equation}
\label{eq:upsilon}
\upsilon_\varphi \coloneqq -\frac{\sum_{k \in \Z} \varphi'(t_0-k)^2}{\overline 
\sigma_\varphi \sigma_\varphi''(t_0)},
\end{equation}
and define the quantities
\begin{align*}
  a(j) &\coloneqq \sqrt{2 \log(2) j},\\
  b(j) &\coloneqq a(j) - \frac{\log (\pi \log 2) + \log j - \tfrac12 \log (1 + 
  \upsilon_\varphi)}{2a(j)},\\
  c(j) &\coloneqq \frac{\overline \sigma_\varphi}{\sigma} 2^{j/2},\\
  x(\gamma) &\coloneqq -\log \left(-\log (1-\gamma)\right).
\end{align*}
We then have the following result on the distribution of the variance term. 

\begin{theorem}
  \label{thm:result}
  Let \(j_n \to \infty,\) \(\gamma_0 \in (0, 1),\) and either:
  \begin{enumerate}
    \item for a wavelet basis on \(\R,\) \(\Gamma_n \coloneqq (0, \gamma_0];\) 
      or
    \item for a wavelet basis on \([0, 1],\) \(\Gamma_n \coloneqq [\gamma_n, 
      \gamma_0],\) where \(\gamma_n \in (0, \gamma_0),\) and \(\gamma_n^{-1} = 
      o(e^{Cj_n})\) for any \(C > 0.\)
  \end{enumerate}
  Then, as \(n \to \infty,\)
  \[\sup_{\gamma \in \Gamma_n} \abs*{ \gamma^{-1}
  \P\left(\norm{\fh(j_n) - \E \fh(j_n)}_\infty > 
  c(j_n)\left(\frac{x(\gamma)}{a(j_n)} + b(j_n)\right)\right) - 1} \to 0.\]
\end{theorem}

While this result is stated for the white noise model, similar results hold 
also in density estimation and regression. In density estimation, \(f\) is a 
density, and we observe
\[X_1, \dots, X_n \iid f.\]
This can be linked to the white noise model using 
\citet[\S4.1]{gine_confidence_2010}. In regression, we have independent 
observations
\[Y_i \sim N(f(x_i), \sigma^2),\]
for \(x_i \coloneqq i/n,\) \(i = 1, \dots, n.\) Regression is known to be 
asymptotically equivalent to white noise, as in \citet{brown_asymptotic_1996}.  
We can thus transfer our result also to these models.

\section{Numerical approximations}
\label{sec:numerics}

To apply our result, we must first verify \autoref{ass:sigma-maximum},
which depends on the function \(\varphi\) and its derivatives. In general, 
\(\varphi\) has no explicit form, but we can approximate it numerically using 
the cascade algorithm. \(\varphi\) satisfies a two-scale relation,
\[\varphi(x) = \sum_{k=0}^{2K-1} u_k^{(0)} \varphi(2x+K-k),\]
for filter coefficients \(u^{(0)}_0, \dots, u^{(0)}_{2K-1} \in \R\) satisfying
\[\sum_{k \text{ odd}} u^{(0)}_k = \sum_{k \text{ even}} u^{(0)}_k = 1,\]
and we can use these filter coefficients to compute an approximation to 
\(\varphi.\)

For \(n \ge 0,\) \(k = 0, \dots, 2K-1,\) set
\[u^{(n+1)}_k \coloneqq 2 \sum_{i=0}^k (-1)^i u^{(n)}_{k-i},\]
and for \(j, n \ge 0,\) \(0 \le k < 2^j(2K-1),\) \(x \in [1-K, K),\)
\begin{align*}
  g^{(n)}_{0, k} &\coloneqq \delta_k, &
  g^{(n)}_{j+1, k} &\coloneqq \sum_{i \in \Z} g_{j, i}^{(n)} u_{k-2i}^{(n)},\\
  f^{(n)}_{j, k} &\coloneqq \sum_{i=0}^n \binom{n}{i} (-1)^i g^{(n)}_{j, k-2^j 
  i}, &
  f^{(n)}_j(x) &\coloneqq f^{(n)}_{j, \lfloor 2^j (x + K - 1) \rfloor}.
\end{align*}
The functions \(f^{(0)}_j\) then converge to a limit function \(f\) defined by 
the \(u^{(0)}_k,\) and the \(f^{(n)}_j\) likewise converge to \(f^{(n)}.\)  
The following theorem bounds the error in this approximation, and is a 
straightforward consequence of results in \citet{rioul_simple_1992}.

\begin{theorem}
  \label{thm:wavelet-approx}
  For integers \(j, n \ge 0,\) set
  \begin{multline*}
    \begin{aligned}
      \alpha_j^{(n)} &\coloneqq 1 - j^{-1} \log_2 \left( \max_{k=0}^{2^j-1} 
      \sum_{i=0}^{2K-2} \abs{g_{j, k+2^j i}^{(n+1)}} \right),\\
      C_j^{(n)} &\coloneqq \left(1-2^{-\alpha_j^{(n)}}\right)^{-1} \left( 
      \max_{l=0}^{j-1} \max_{k=0}^{2^l(2K-1)-1} 2^{(\alpha_j^{(n)} - 1)l} 
      \abs{f_{l, k}^{(n+1)}}\right)
    \end{aligned}
    \\ \left(\max_{m=0, 1} \sum_{k=0}^{K-1} \abs*{ \sum_{i=0}^k 
    u_{2i+m}^{(n)}-1}\right).
  \end{multline*}
  If \(\alpha_j^{(0)} > 0\) for some \(j,\) the functions \(f_j^{(0)}\) 
  converge in \(L^\infty\) to a function \(f:[1-K, K) \to \R\) satisfying
  \[f(x) = \sum_{k=0}^{2K-1} u_k^{(0)} f(2x+K-k).\]
  If also \(\alpha_j^{(n)} > 0\) for some \(j,\) and \(n > 0,\) then \(f\) is 
  \(n\)-times-differentiable, and the \(f_j^{(n)}\) converge in \(L^\infty\) 
  to \(f^{(n)}.\)
  For \(n \ge 0,\) the approximations \(f_j^{(n)}\) converge at a rate
  \[\norm{f_j^{(n)} - f^{(n)}}_{L^\infty} \le C_j^{(n)} 2^{-j 
  \alpha_j^{(n)}}.\]

  Furthermore, given integers \(j \ge 0,\) \(a \le b,\) set
  \[I \coloneqq 2^{-j}[a, b+1) + \Z, \qquad J(l) \coloneqq \left[\lfloor 
  2^{l-j} a \rfloor - 2K + 2, \lfloor 2^{l-j}b \rfloor\right] + 2^l \Z.\]
  Then, on \(I \cap [1-K, K)\):
  \begin{enumerate}
    \item for \(l \ge j,\) the values of \(f^{(n)}_l\) depend on 
      \(g_{j,k}^{(n)}\) only for \(k \in J(j);\) and
    \item the above results hold also for quantities \(\alpha_j^{(n)}(I)\) and 
      \(C_j^{(n)}(I)\) defined similarly, taking maxima over 
      \(g_{l,k}^{(n+1)}\) and \(f_{l,k}^{(n+1)}\) only for \(k \in J(l).\)
  \end{enumerate}
\end{theorem}

The function \(\varphi\) may be defined as the limit of this procedure 
\citep[\S6.5]{daubechies_ten_1992}. We may thus compute upper and lower bounds 
on \(\varphi\) and its derivatives.  Note that, while we could obtain values 
of the derivatives by finite differencing, this would be numerically unstable, 
and lead to poor bounds; the above procedure provides good bounds on all 
derivatives of \(\varphi.\)

To verify \autoref{ass:sigma-maximum}, and to compute the constants \(\sfb^2\) 
and \(\upsilon_\varphi,\) we must use these bounds to control the function 
\(\sigma^2_\varphi,\) and its derivatives. Doing so over the whole of \([0, 
1]\) requires memory exponential in \(j,\) which quickly becomes infeasible.  
However, once we have approximated \(\sigma^2_\varphi\) well enough to know 
that its maxima lie in some interval \(I,\) we can exploit the local nature of 
the cascade algorithm, and its bounds, to approximate \(\sigma^2_\varphi\) 
only over \(I.\) As the resolution \(j\) increases, so does the accuracy with 
which we can locate the maxima, ensuring our memory costs remain manageable.

In our implementation, we choose \(I\) to be the smallest interval containing 
all points \(t\) for which the bounds on \(\sigma^2_\varphi,\) and its
derivative, are consistent with:
\begin{enumerate}
  \item \(\sigma_\varphi^2(t) = \sup_{s \in [0, 1]} \sigma_\varphi^2(s);\) and
  \item \( (\sigma_\varphi^2)'(t) = 0.\)
\end{enumerate}
Note that to ensure efficiency, we must allow choices of \(I\) which wrap 
around the edges of \([0, 1];\) in other words, we must allow \(I\) to be any 
interval on the torus. If we find an interval \(I\) containing all maxima of 
\(\sigma^2_\varphi,\) with the property that \(\sigma_\varphi'' \le 
-\varepsilon < 0\) on \(I,\) we may conclude \autoref{ass:sigma-maximum} is 
satisfied. We have thus described \autoref{alg:verify}.

\begin{algorithm}
\caption{Verify assumption and compute constants}
\label{alg:verify}
\begin{algorithmic}
  \STATE \(I \gets [0, 1]\)
  \STATE \(j \gets 0\)
  \REPEAT
  \STATE calculate approximations \(\varphi^{(n)}_j\) to \(\varphi^{(n)}\) on 
  \(I,\) \(n = 0, 1, 2\)
  \STATE deduce bounds on \((\sigma^2_\varphi)^{(n)}\) on \(I,\) \(n = 0, 1, 
  2\)
  \STATE deduce bounds on \(\sfb^2\) and \(\upsilon_\varphi\)
  \STATE \(I \gets\) smallest interval known to contain all maxima of 
  \(\sigma^2_\varphi\)
  \STATE \(j \gets j + 1\)
  \UNTIL{desired accuracy reached}
  \IF{\(\sigma_\varphi''\) bounded below zero on \(I\)}
  \STATE \autoref{ass:sigma-maximum} is verified
  \ENDIF
\end{algorithmic}
\end{algorithm}

To obtain high accuracy, the computation of the filter coefficients \(u_k,\) 
and subsequent approximations, must be performed using variable-precision 
arithmetic; the rounding error in these computations must likewise be 
controlled with interval arithmetic. We satisfy these requirements by 
implementing the above algorithm in the computer algebra system Mathematica.  
For Daubechies wavelets and symlets, \(N = 6, \dots, 20,\) we find that 
\autoref{ass:sigma-maximum} is indeed satisfied, and obtain accurate values of 
\(\sfb^2\) and \(\upsilon_\varphi,\) given in \autoref{tab:wavelet-values}.

\begin{table}
  \centering
  \begin{tabular}{rllll}
    \toprule
    & \multicolumn{2}{c}{Daubechies} & \multicolumn{2}{c}{Symlet}    \\
    \midrule
    \multicolumn{1}{c}{\(N\)} &
    \multicolumn{1}{c}{\(\overline \sigma^2_\varphi\)} & 
    \multicolumn{1}{c}{\(\upsilon_\varphi\)} &
    \multicolumn{1}{c}{\(\overline \sigma^2_\varphi\)} & 
    \multicolumn{1}{c}{\(\upsilon_\varphi\)} \\
    \midrule
    6    &   1.251 716   &   0.221 993   &   1.361 961   &   0.106 518   \\
    7    &   1.276 330   &   0.197 328   &   1.253 835   &   0.248 681   \\
    8    &   1.250 928   &   0.266 316   &   1.286 722   &   0.173 642   \\
    9    &   1.222 637   &   0.275 519   &   1.232 334   &   0.302 351   \\
    10   &   1.199 772   &   0.391 629   &   1.243 114   &   0.255 337   \\
    11   &   1.195 384   &   0.415 019   &   1.209 007   &   0.324 200   \\
    12   &   1.189 984   &   0.445 388   &   1.215 480   &   0.335 022   \\
    13   &   1.182 351   &   0.460 792   &   1.195 567   &   0.385 147   \\
    14   &   1.172 690   &   0.510 179   &   1.195 969   &   0.405 884   \\
    15   &   1.165 335   &   0.553 767   &   1.184 307   &   0.446 419   \\
    16   &   1.159 678   &   0.594 027   &   1.181 901   &   0.465 670   \\
    17   &   1.154 955   &   0.621 941   &   1.174 105   &   0.496 485   \\
    18   &   1.150 103   &   0.652 913   &   1.170 871   &   0.520 228   \\
    19   &   1.145 393   &   0.686 434   &   1.164 974   &   0.551 765   \\
    20   &   1.141 050   &   0.722 113   &   1.161 837   &   0.571 150   \\
    \bottomrule
  \end{tabular}
  \caption{Computed values of constants}
  \label{tab:wavelet-values}
\end{table}

\subsection*{Acknowledgements}

We would like to thank Richard Nickl for his valuable comments and 
suggestions.

\appendix

\section{Proofs}
\label{app:proofs}

We will need the following result, which is a version of Theorem 1 in 
\citet{husler_extremes_1999}. The result concerns the maxima of centred 
Gaussian processes whose variance functions are periodic; such processes are 
called {\em cyclostationary}. In \citeauthor{husler_extremes_1999}'s original 
result, the maxima of a sequence of processes was shown to converge to a 
Gumbel random variable. In our result, we will specialise to a single process, 
and show this convergence occurs uniformly.

\begin{lemma}
  \label{lem:cyclostationary-process}

  Let \(T = T(n) \to \infty\) as \(n \to \infty.\) In the notation of 
  \citet{husler_extremes_1999}, let (A1)--(A3) and (B1)--(B4) hold, for a 
  fixed process \(X_n(t) = X(t),\) not depending on \(n.\) Further let \(\alpha = \beta,\) and 
  let \citeauthor{husler_extremes_1999}'s condition (1) hold.  Define
  \[u(\tau) = \sigma_n \mu^{-1}(\tau/m_T).\]
  Then for any \(\tau_0 > 0,\) we have
  \[\sup_{\tau \in (0, \tau_0]} \abs*{ \frac{\P(M_n(T) > u(\tau))}{1 - 
  e^{-\tau}} - 1} \to 0\]
  as \(n \to \infty.\)
\end{lemma}

\begin{proof}
  Our argument proceeds as in the proof of Theorem 1 in
  \citet{husler_extremes_1999}.  Without loss of generality, we may assume 
  that \(\sigma_n = 1.\) For \(\tau \le \tau_0,\) \(u(\tau) \ge u(\tau_0) \to 
  \infty,\) and by definition
  \[m_T \mu(u(\tau)) = \tau.\]
  The approximation errors in parts (i) and (ii) of 
  \citeauthor{husler_extremes_1999}'s proof are thus \(O(g(S)\tau)\) and 
  \(O(\rho_c\tau)\) respectively.  In part (iii), we note that
  \[u(\tau)^2 = 2 \log (T/\tau) - \log \log (T/\tau) + O(1),\]
  so \citeauthor{husler_extremes_1999}'s term (4) is of order
  \begin{multline*}
    \tau^{1+\eta}(T/\tau)^{1+\eta}u(\tau)^{2/\alpha}\exp\{-u(\tau)^2/(1+\gamma)\}
    \\= \tau^{1+\eta} \exp\{-[(1-\gamma)/(1+\gamma) - \eta]\log (T/\tau) + 
    o(\log(T/\tau))\} = o(\tau),
  \end{multline*}
  and term (5) is of order
  \begin{multline*}
    \tau^2 (T/\tau)^2 \delta(T^\eta) \exp \{-(2\log (T/\tau) - \log \log 
    (T/\tau))(1 - \delta(T^\eta)\}
    \\= O(\tau^2 \delta(T^\eta) \log (T/\tau)) = o(\tau).
  \end{multline*} 
  In \citeauthor{husler_extremes_1999}'s final display, we may thus write
  \[\P(M_n(T) \le u(\tau)) = \exp\{-(1+o(1))\tau\} + o(\tau).\]

  As the process \(X(t)\) does not depend on \(n,\) the error in each of these 
  approximations depends only on \(u = u(\tau),\) and the above limits hold as \(u \to 
  \infty.\) (This can be seen from the precise form of the errors, as given in 
  \citealp[\S3.1]{piterbarg_linear_1994}, and in 
  \citeauthor{husler_extremes_1999}'s proof.) Since \(u\) is decreasing in 
  \(\tau,\) the limits are therefore uniform in \(\tau\) small.

  Consider the function
  \[f(x, y; \tau) \coloneqq \log\left(\frac{1 - \exp(-(1 + x)\tau)}{\tau} + 
  y\right),\]
  defined on \(0 \le \tau \le \tau_0,\) \(\abs{x} \le \tfrac12,\) \(\abs{y} 
  \le \tfrac12 (1 - \exp(-\tfrac12\tau_0))/\tau_0.\) The derivatives
  \[\frac{\delta f}{\delta x} = \frac{\exp(-(1 + x)\tau)}{\exp f}, \qquad 
  \frac{\delta f}{\delta y} = \frac{1}{\exp f}\]
  are finite, and continuous in \(x,\) \(y\) and \(\tau,\) so by the mean 
  value inequality, for \(n\) large,
  \begin{align*}
    \log \left(\frac{\P(M_n(T) > u(\tau))}{\tau}\right) &= f(o(1), o(1); \tau)
    \\&= f(0, 0; \tau) + o(1) \\
    &= \log\left(\frac{1-e^{-\tau}}{\tau}\right) + o(1).
  \end{align*}
  As the above limits are uniform in \(\tau \le \tau_0,\) the result follows.
\end{proof}

We now apply this result to a cyclostationary process, composed of scaling 
functions \(\varphi,\) which we can use to model the variance of estimators 
\(\fh(j_n).\)

\begin{lemma}
  \label{lem:wavelet-process}
  Define the cyclostationary Gaussian process
  \[X(t) \coloneqq \sfb^{-1} \sum_{k \in \Z} \varphi(t - k) Z_k, \qquad Z_k 
  \iid N(0, 1).\]
  For any \(\gamma_0 \in (0, 1),\) \(j_n \to \infty,\)
  \[\sup_{\gamma \in (0, \gamma_0]} \abs*{ \gamma^{-1}
  \P\left(\sup_{t \in [0, 2^{j_n}]} \abs{X(t)} > \frac{x(\gamma)}{a(j_n)} + 
  b(j_n)\right) - 1 } \to 0\]
  as \(n \to \infty.\)
\end{lemma}

\begin{proof}
  For fixed \(\gamma,\) the result is a consequence of Theorem 2 in 
  \citet{gine_confidence_2010}; the statement uniform over \((0, \gamma_0]\) 
  follows, replacing Theorem 1 of \citet{husler_extremes_1999} in 
  \citeauthor{gine_confidence_2010}'s proof with 
  \autoref{lem:cyclostationary-process}.  The conditions of 
  \citeauthor{gine_confidence_2010}'s theorem are satisfied by 
  \hyperref[ass:wavelet-basis]{Assumptions \ref{ass:wavelet-basis}} and 
  \ref{ass:sigma-maximum}, as follows.
  \begin{enumerate}
    \item \(X\) has almost-sure derivative
      \[X'(t) \coloneqq \sfb^{-1} \sum_{k \in \Z} \varphi'(t-k) Z_k,\]
      so is continuous. \(X'\) is also the mean square derivative:
      \begin{multline*}
        h^{-1}\E[(X(t+h)-X(t)-hX'(t))^2]\\
          = h^{-1} \sum_{k \in \Z} \left(\varphi(t-h-k) - \varphi(t-k) 
          -h\varphi'(t-k)\right)^2,
      \end{multline*}
      which tends to 0 as \(h \to 0,\) since the sum has finitely many 
      non-zero terms.

    \item For \(i=0, 1,\) define functions \(f_i(x) \coloneqq x^i\) on \([0, 
      1],\) having wavelet expansions
      \[f_i = \sum_k \alpha_{J, k}^i \varphi_k + \sum_{j>J} \sum_k 
      \beta_{j,k}^i \psi_{j,k}\]
      in our wavelet basis on \([0, 1],\) for some \(J \ge j_0,\) \(2^J \ge 
      6K.\)
      As \(\psi\) is twice continuously differentiable, and \(\varphi\) and 
      \(\psi\) have compact support, by Corollary 5.5.4 in 
      \citet{daubechies_ten_1992}, \(\psi\) has at least two vanishing moments. 
      Thus
      \[\beta_{j,k}^i = \langle x^i, \psi_{j,k} \rangle = 0,\]
      and
      \[f_i(t) = \sum_k \alpha_k^i \varphi_{J,k}(t).\]

      For \(t \in [0, 1],\) let \(v(t)\) denote the vector 
      \((\varphi_{J,k}(t)) \in \R^{2^J},\) so \(f_i(t) = \langle \alpha^i, 
      v(t) \rangle.\) Given \(s \ne t,\) we have
      \begin{align*}
        \langle \alpha^0, v(s) \rangle &= 1 = \langle \alpha^0, v(t) 
        \rangle,\\
        \langle \alpha^1, v(s) \rangle &= s \ne t = \langle \alpha^1, v(t) 
        \rangle,
      \end{align*}
      so the vectors \(v(s),\) \(v(t)\) are linearly independent.
      
      For \(s, t \in \R,\) define
      \[r_X(s, t) \coloneqq \Cov[X(s), X(t)], \qquad \sigma^2_X(t) \coloneqq 
      \Var[X(t)] = r_X(t, t).\]
      Then, if \(s, t \in [-K, K],\)
      \begin{align*}
        r_X(s, t) &= \sfb^{-2} \sum_{k \in \Z} \psi(s-k)\psi(t-k)\\
        &= \sfb^{-2} 2^{-J} \langle v(\tfrac12 + 2^{-J}s), v(\tfrac12 + 
        2^{-J}t)\rangle,
      \end{align*}
      so by Cauchy-Schwarz,
      \[r_X(s, t)^2 < \sigma_X^2(s) \sigma_X^2(t).\]
      If \(s, t \in [k-K, k+K]\) for some \(k \in \Z,\) the same applies by 
      cyclostationarity. If not, then as \(\varphi\) is supported on \([1-K, 
      K],\) we have \(r_X(s, t) = 0.\) However, for any \(t \in [0, 1],\) 
      \(\langle \alpha^1, v(\tfrac12 + 2^{-J}t) \rangle = 1,\) so
      \[\sigma_X^2(t) = \sfb^{-2} 2^{-J} \norm{v(t)}^2 > 0,\]
      and by cyclostationarity the same holds for all \(t \in \R.\) We thus 
      again obtain
      \[r_X(s, t)^2 < \sigma_X^2(s) \sigma_X^2(t).\]

    \item We have
      \[\sigma_X^2(t) = \sfb^{-2} \sigma_\varphi^2(t),\]
      so by \autoref{ass:sigma-maximum}, \(\sup_{t \in [0,1]} \sigma_X^2(t) = 
      1,\) and this maximum is attained at a unique \(t_0 \in [0, 1).\) If 
      \(t_0 \in (0, 1),\) this satisfies the conditions of the theorem 
      directly; if not we may proceed as in Proposition 9 of 
      \citet{gine_confidence_2010}.  \(\sigma_\varphi^2\) is twice 
      differentiable,
      \[2\sfb \sigma_X'(t_0) = (\sigma_\varphi^2)'(t_0) 
      \sigma_\varphi^2(t_0)^{-1/2} = 0,\]
      and
      \begin{align*}
        2 \sfb \sigma_X''(t_0) &= (\sigma_\varphi^2)''(t_0) 
        \sigma_\varphi^2(t_0)^{-1/2} - \tfrac12 (\sigma_\varphi^2)'(t_0)^2 
        \sigma_\varphi^2(t_0)^{-3/2}\\
        &= (\sigma_\varphi^2)''(t_0) \sigma_\varphi^2(t_0)^{-1/2} < 0.
      \end{align*}
      Finally, let \(v'(t)\) denote the vector \((\varphi'_{J,k}(t)) \in 
      \R^\Z.\) Then for \(t \in [0, 1],\)
      \[\langle \alpha^1, v'(t) \rangle = f_1'(t) = 1,\]
      so
      \begin{align}
      \label{eq:upsilon-numerator}
      \E[X'(t_0)^2]
      &= \sfb^{-2} \sum_{k \in \Z} \varphi'(t_0-k)^2\\
      \notag &= \sfb^{-2} 2^{-J} \norm{v'(\tfrac12 + 2^{-J}t_0)}^2 > 0.
      \end{align}

    \item Since \(\varphi\) has support \([1-K, K],\)
      \[\sup_{s,t:\abs{s-t} \ge 2K-1} \abs{r_X(s, t)} = 0. \qedhere\]
  \end{enumerate}
\end{proof}

We may now bound the variance of \(\fh(j_n).\) We will show that the variance 
process is distributed as the process \(X\) from the above lemma, so can be 
controlled similarly.

\begin{proof}[Proof of \autoref{thm:result}]
  Let \(I_n \coloneqq [0, 2^{j_n}].\) The process
  \[X_n(t) \coloneqq \frac{\fh(j_n) - \fb(j_n)}{c(j_n)}(2^{-j_n}t), \qquad t 
  \in I_n,\]
  is distributed as
  \[\sfb^{-1}2^{-j_n/2} \left( \sum_{k\in\Z} Z_{j_0,k} 
  \varphi_{j_0,k}(2^{-j_n}t) + \sum_{j = j_0 + 1}^{j_n} \sum_{k\in \Z} Z_{j,k} 
  \psi_{j,k}(2^{-j_n}t) \right),\]
  for \(Z_{j,k} \iid N(0, 1),\) so by an orthogonal change of basis, as
  \[\sfb^{-1} 2^{-j_n/2} \sum_{k\in\Z} Z_k \varphi_{j_n,k}(2^{-j_n}t), \qquad 
  Z_k \iid N(0, 1).\]
  In case (i), \(X_n\) is distributed as the process \(X\) from 
  \autoref{lem:wavelet-process}, so we are done.

In case (ii), set \(J_n \coloneqq [2K, 2^{j_n} - 2K],\) and \(K_n \coloneqq 
I_n \setminus J_n.\) On \(J_n,\) \(X_n\) is distributed as the process \(X\) 
from \autoref{lem:wavelet-process}, and we have
  \begin{align*}
    \P\left(\sup_{t \in J_n} \abs{X_n(t)} > u\right)
    &\le \P\left(\sup_{t \in I_n} \abs{X_n(t)} > u\right)\\
    &\le \P\left(\sup_{t \in J_n} \abs{X_n(t)} > u\right) +
    \P\left(\sup_{t \in K_n} \abs{X_n(t)} > u\right),
  \end{align*}
  so for \(u_n(j_n) \coloneqq x(\gamma_n)/a(j_n) + b(j_n),\)
  \begin{multline*}
    \abs*{\P\left(\sup_{t \in I_n} \abs{X_n(t)} > u_n(j_n)\right) - 
    \P\left(\sup_{t \in J_n} \abs{X(t)} > u_n(j_n)\right)}\\
    \begin{aligned}
      &\le \P\left(\sup_{t \in K_n} \abs{X_n(t)} > u_n(j_n)\right)\\
      &\le 8K(1 - \Phi(Cu_n(j_n)))\\
      &\lesssim e^{-C^2u_n(j_n)^2/2}/u_n(j_n),
    \end{aligned}
  \end{multline*}
  with a constant \(C > 0\) depending on \(\varphi.\) This term is 
  \(o(\gamma_n),\) so the result follows by \autoref{lem:wavelet-process}, 
  applied to the process \(X\) on \(J_n.\)
\end{proof}

\section{Source code}
\label{app:source}

The following program implements \autoref{alg:verify} in Mathematica 8 or 
above. Note that we bound \(\upsilon_\varphi\) by bounding the numerator and 
denominator of \eqref{eq:upsilon} separately over \(I.\) By 
\eqref{eq:upsilon-numerator}, the numerator is positive; to bound 
\(\upsilon_\varphi\) inside \((0, \infty),\) we must therefore bound 
\(\sigma_\varphi''\) below zero. To verify \autoref{ass:sigma-maximum}, it is 
thus sufficient that we establish a finite positive value of 
\(\upsilon_\varphi.\)

{\scriptsize 
\begin{verbatim}
Main::usage =
  "Main[w, p, d, jmax, kmax] = bounds on the parameters sigma^2_phi and\n" <>
  "upsilon_phi for the Wavelet w, performing computations using p digits\n" <>
  "of accuracy, stopping once accurate to d digits, after jmax steps,\n" <>
  "after evaluating phi at 2^kmax locations simultaneously, or after\n" <>
  "the results become indeterminate due to lack of precision.\n\n" <>
  "Main[DaubechiesWavelet[6], 200, 6, 100, 12] = \n" <>
  "  {Interval[{1.251716, 1.251716}], Interval[{0.221993, 0.221993}]}"
Main[wavelet_, precision_, digits_, jmax_, kmax_] :=
  Cascade[Filter[wavelet, precision], 2, PhiParams[digits, jmax, kmax]]

Filter::usage =
  "Filter[w] = the cascade filter of a Wavelet w."
Filter[w_, p_:MachinePrecision] :=
  2 Transpose[WaveletFilterCoefficients[w, WorkingPrecision -> p]][[2]]

PhiParams::usage =
  "PhiParams[d, jm, km] = a callback for use with Cascade, which\n" <>
  "computes the parameters sigma^2_phi and upsilon_phi to d d.p., with\n" <>
  "resource limits jm, km."
PhiParams[d_, jm_, km_] :=
  Params[Function[{w, j, offset, phi},
  Module[{PhiMN, s2, s2max, ssp, sp2, sspp, s22b, s22a, ba, zoom},
  PhiMN[m_, n_] := Plus @@ If[m == n, phi[[n+1]]^2, phi[[m+1]] phi[[n+1]]];

  s2 = PhiMN[0, 0]; ssp = PhiMN[0, 1];
  s2max = IntervalMax[s2];
  sp2 = PhiMN[1, 1]; sspp = PhiMN[0, 2];
  s22b = IntervalRange[sp2];
  s22a = -IntervalRange[sspp+sp2];
  ba = s22b / s22a;

  zoom = Thread[(!TrueQ[# < s2max] & /@ s2) && (!TrueQ[# != 0] & /@ ssp)];
  {{s2max, ba}, zoom}]], d, jm, km]

Params::usage =
  "Params[Callback, digits, jmax, kmax] = a callback function for\n" <>
  "use with Cascade. The function will pass its arguments to Callback,\n" <>
  "which should return estimates of parameters, and an array of indices k\n"<>
  "to restrict to. The function will terminate the cascade algorithm if\n" <>
  "the returned parameters are accurate to digits d.p., the resource\n"<>
  "limits jmax, kmax are reached, or the computation is indeterminate.\n" <>
  "Otherwise, it will restrict the algorithm to an interval I containing\n" <>
  "the requested indices, and continue."
Params[Callback_, digits_, jmax_, kmax_] :=
  Function[{w, j, offset, phi}, Module[{ret, zoom, len, i, l},
  {ret, zoom} = Callback[w, j, offset, phi];
  len = Dimensions[phi][[3]];
  Which[
  MemberQ[ret, Indeterminate, Infinity],
    Print["indeterminate, j = ", j];
    ret = If[MemberQ[#, Indeterminate, Infinity],
    Interval[{-Infinity, Infinity}], #] & /@ ret;
    Prepend[ret, True],
  (And @@ (IntervalAccurate[#, digits] &) /@ ret) ||
  (j == jmax && (Print["hit jmax"]; True)) ||
  (len >= 2^kmax && (Print["hit kmax, j = ", j]; True)),
    Prepend[ret, True],
  len == 2^j,
    {l, i} = LongestSubsequence[Join[zoom, zoom], # == False &];
    If[l < 2^(j-1), {False}, Prepend[Mod[#-1, 2^j]+1 & /@ {i+l, i-1}, False]],
  True,
    Prepend[Through[{Min, Max}[Position[zoom, True]]], False]]]]

IntervalAccurate::usage =
  "IntervalAccurate[i, d] = True if the Interval i is accurate to d d.p."
IntervalAccurate[int_, digits_] :=
  Equal @@ Round[Through[{Min, Max}[int]], 10^(-digits)]
  
IntervalRange::usage =
  "IntervalRange[is] = an Interval giving the range of the Intervals is."
IntervalRange[ints_] :=
  Interval @ Through[{Min, Max}[ints]]
  
IntervalMax::usage =
  "IntervalMax[is] = an Interval giving the maximum of the Intervals is."
IntervalMax[ints_] :=
  Interval @ Map[Max, Transpose[ints /. Interval -> Identity]]

LongestSubsequence::usage =
  "LongestSubsequence[x, p] = {l, i}, where l is the length of the\n" <>
  "longest consecutive subsequence of x whose members satisfy p, and i\n" <>
  "is the index of its first member."
LongestSubsequence[x_, crit_] := 
  Block[{$RecursionLimit = Infinity}, Module[{m, i, l},
  m = LengthWhile[x, crit];
  If[m == Length[x], {m, 1},
    {l, i} = LongestSubsequence[Drop[x, m + 1], crit]; 
    If[l > m, {l, i + m + 1}, {m, 1}]]]]

Cascade::usage =
  "Cascade[w, n, Callback] runs the cascade algorithm with filter w,\n" <>
  "bounding f and its first n derivatives.\n\n" <>
  "The function Callback[w, j, offset, phi] should return a list ret.\n" <>
  "If First[ret] is True, cascade will halt, returning Rest[ret].\n" <>
  "If First[ret] is False, the algorithm will continue, and if\n" <>
  "Rest[ret] = {a, b} is given, future evaluation of f will be\n" <>
  "restricted to the indices a,...,b of phi.\n"
Cascade[w_, n_, Callback_] :=
  Module[{offset, u, j, k, g, f, fs, eps, phi, a, b, pts}, 
  offset = 0; j = 0; k = Length[w]/2;
  u = BuildU[w, n+1]; g = InitG[k, n+1]; fs = {BuildF[g]};
  While[True,
    g = StepG[g, u]; f = BuildF[g]; eps = BoundError[fs, g, u];
    phi = ApplyError[f, eps, k]; AppendTo[fs, f]; offset *= 2; j += 1;
    ret = Callback[w, j, offset, phi]; If[First[ret], Return[Rest[ret]]];
    {a, b} = If[Length[ret] > 1, Rest[ret], {1, Dimensions[phi][[3]]}];
    If[a > b && Dimensions[phi][[3]] == 2^j,
      {fs, g} = WrapFG[fs, g, k]; b = b + 2^j; offset -= 2^j];
    g = Take[g, All, All, {a, 2(k-1)+b}];
    fs = RestrictF[fs, offset, a, b, k]; offset += a-1]]

BuildU::usage =
  "BuildU[u^(0), n] = u^(0:n)"
BuildU[w_, n_] :=
  Map[Function[un, Table[un[[m;;All;;2]], {m, 2}]],
  With[{signs = Table[(-1)^i, {i, Length[w]}]},
  NestList[2 signs Accumulate[signs #] &, w, n]]]

InitG::usage =
  "InitG[k, n] = g_0^(0:n)"
InitG[k_, n_] :=
  ConstantArray[Reverse @ IdentityMatrix[2k-1], n+1]

StepG::usage =
  "StepG[g_j^(0:n), u^(0:n)] = g_{j+1}^(0:n)"
StepG[g_, u_] :=
  MapThread[Function[{gn, un}, Function[gni,
  Riffle @@ (ListConvolve[#, gni] &) /@ un] /@ gn], {g, u}]

BuildF::usage =
  "BuildF[g_j^(0:n)] = f_j^(0:n)"
BuildF[g_] :=
  MapIndexed[Function[{gn, part}, Module[{n, bn},
  n = First[part]-1; bn = Table[Binomial[n, i](-1)^i, {i, 0, n}];
  Transpose[ListConvolve[bn, #, 1, 0] & /@ Transpose @ gn]]], g]

BoundAlpha::usage =
  "BoundAlpha[g_j^(0:n), j] = alpha_j^(0:n-1)"
BoundAlpha[g_, j_] :=
  1-Log[2, Max @@ Plus @@ Abs @ #]/j & /@ Rest @ g

BoundC::usage =
  "BoundC[f_{0:j-1}^(0:n), u^(0:n), alpha_j^(0:n-1)] = C_j^(0:n-1)"
BoundC[f_, u_, alpha_] :=
  MapThread[Function[{fn, un, an}, If[an <= 0, Infinity,
  (Max @ MapIndexed[2^((an-1)(First[#2]-1)) Abs[#1] &, fn])
  (Max @@ Plus @@ (Abs[Accumulate[#]-1] &) /@ un)/(1-2^(-an))]], 
  {Rest @ Transpose @ f, Most @ u, alpha}]

BoundError::usage =
  "BoundError[f_{0:j-1}^(0:n), g_j^(0:n), u^(0:n)] = eps_j^(0:n-1)"
BoundError[f_, g_, u_] := Module[{j, alpha, c}, 
  j = Length[f]; alpha = BoundAlpha[g, j]; c = BoundC[f, u, alpha];
  MapThread[Function[{cn, an}, cn 2^(-j an)], {c, alpha}]]

ApplyError::usage =
  "ApplyError[f_j^(0:n), eps_j^(0:n-1), k] = intervals bounding f_j^(0:n-1)"
ApplyError[f_, eps_, k_] :=
  MapThread[Function[{fn, en}, 
  Map[Interval[{#-en, #+en}] &, fn, {2}]], 
  {Most @ Take[f, All, All, {2k-1, -1}], eps}]

WrapFG::usage =
  "WrapFG[f_{0:j}^(0:n), g_j^(0:n), k] = f and g wrapped around at integers"
WrapFG[f_, g_, k_] :=
  Module[{Wrap}, Wrap[x_] :=
    Map[Function[xn, Map[Function[xni, Join[xni[[1]], xni[[2]][[2k-1;;]]]],
    Partition[xn, 2, 1, {-1, 1}, {ConstantArray[0, Dimensions[x][[3]]]}]]],
    x]; {Wrap /@ f, Wrap @ g}]

RestrictF::usage =
  "RestrictF[f_{0:j}^(0:n), offset, a, b, k] = f_{0:j}^(0:n)|offset+[a, b]"
RestrictF[fs_, offset_, a_, b_, k_] :=
  MapIndexed[Function[{fj, part},
  Module[{scale, start, end}, scale = 2^(Length[fs]-First[part]);
  start = Floor[(offset+a-1)/scale]-Floor[offset/scale]+1;
  end = 2(k-1)+Floor[(offset+b-1)/scale]-Floor[offset/scale]+1;
  Take[fj, All, All, {start, end}]]], fs]
\end{verbatim}}

\bibliographystyle{abbrvnat}
{\footnotesize \bibliography{sbrw}}

\begin{thebibliography}{15}
\providecommand{\natexlab}[1]{#1}
\providecommand{\url}[1]{\texttt{#1}}
\expandafter\ifx\csname urlstyle\endcsname\relax
  \providecommand{\doi}[1]{doi: #1}\else
  \providecommand{\doi}{doi: \begingroup \urlstyle{rm}\Url}\fi

\bibitem[Bickel and Rosenblatt(1973)]{bickel_global_1973}
Bickel  P J and Rosenblatt  M.
\newblock On some global measures of the deviations of density function
  estimates.
\newblock \emph{The Annals of Statistics}, 1:\penalty0 1071{\textendash}1095,
  1973.

\bibitem[Brown and Low(1996)]{brown_asymptotic_1996}
Brown  L D and Low  M G.
\newblock Asymptotic equivalence of nonparametric regression and white noise.
\newblock \emph{The Annals of Statistics}, 24\penalty0 (6):\penalty0
  2384{\textendash}2398, 1996.
\newblock
  \href{http://dx.doi.org/10.1214/aos/1032181159}{doi:10.1214/aos/1032181159}

\bibitem[Bull(2011)]{bull_honest_2011}
Bull  A D.
\newblock Honest adaptive confidence bands and self-similar functions.
\newblock October 2011.
  \href{http://arxiv.org/abs/1110.4985}{{arXiv:1110.4985}}

\bibitem[Chyzak et~al.(2001)Chyzak, Paule, Scherzer, Schoisswohl, and
  Zimmermann]{chyzak_construction_2001}
Chyzak  F, Paule  P, Scherzer  O, Schoisswohl  A, and Zimmermann  B.
\newblock The construction of orthonormal wavelets using symbolic methods and a
  matrix analytical approach for wavelets on the interval.
\newblock \emph{Experimental Mathematics}, 10\penalty0 (1):\penalty0
  67{\textendash}86, 2001.

\bibitem[Cohen et~al.(1993)Cohen, Daubechies, and Vial]{cohen_wavelets_1993}
Cohen  A, Daubechies  I, and Vial  P.
\newblock Wavelets on the interval and fast wavelet transforms.
\newblock \emph{Applied and Computational Harmonic Analysis}, 1\penalty0
  (1):\penalty0 54{\textendash}81, 1993.
\newblock
  \href{http://dx.doi.org/10.1006/acha.1993.1005}{doi:10.1006/acha.1993.1005}

\bibitem[Daubechies(1992)]{daubechies_ten_1992}
Daubechies  I.
\newblock \emph{Ten lectures on wavelets}, volume~61 of \emph{{CBMS-NSF}
  Regional Conference Series in Applied Mathematics}.
\newblock Society for Industrial and Applied Mathematics {(SIAM)},
  Philadelphia, {PA}, 1992.

\bibitem[Gin\'{e} et~al.(2011)Gin\'{e}, G\"{u}nt\"{u}rk, and
  Madych]{gine_periodized_2011}
Gin\'{e}  E, G\"{u}nt\"{u}rk  C S, and Madych  W R.
\newblock On the periodized square of {$L^2$} cardinal splines.
\newblock \emph{Experimental Mathematics}, 20\penalty0 (2):\penalty0 177--188,
  2011.

\bibitem[Gin\'{e} and Nickl(2010)]{gine_confidence_2010}
Gin\'{e}  E and Nickl  R.
\newblock Confidence bands in density estimation.
\newblock \emph{The Annals of Statistics}, 38\penalty0 (2):\penalty0
  1122{\textendash}1170, 2010.
\newblock \href{http://dx.doi.org/10.1214/09-AOS738}{doi:10.1214/09-AOS738}

\bibitem[H\"{a}rdle et~al.(1998)H\"{a}rdle, Kerkyacharian, Picard, and
  Tsybakov]{hardle_wavelets_1998}
H\"{a}rdle  W, Kerkyacharian  G, Picard  D, and Tsybakov  A.
\newblock \emph{Wavelets, approximation, and statistical applications}, volume
  129 of \emph{Lecture Notes in Statistics}.
\newblock {Springer-Verlag}, New York, 1998.

\bibitem[H\"{u}sler(1999)]{husler_extremes_1999}
H\"{u}sler  J.
\newblock Extremes of {G}aussian processes, on results of {P}iterbarg and
  {S}eleznjev.
\newblock \emph{Statistics \& Probability Letters}, 44\penalty0 (3):\penalty0
  251{\textendash}258, 1999.
\newblock
  \href{http://dx.doi.org/10.1016/S0167-7152(99)00016-4}{doi:10.1016/S0167-715%
2(99)00016-4}

\bibitem[H\"{u}sler et~al.(2003)H\"{u}sler, Piterbarg, and
  Seleznjev]{husler_convergence_2003}
H\"{u}sler  J, Piterbarg  V, and Seleznjev  O.
\newblock On convergence of the uniform norms for {G}aussian processes and
  linear approximation problems.
\newblock \emph{The Annals of Applied Probability}, 13\penalty0 (4):\penalty0
  1615{\textendash}1653, 2003.
\newblock
  \href{http://dx.doi.org/10.1214/aoap/1069786514}{doi:10.1214/aoap/1069786514}

\bibitem[Piterbarg and Seleznjev(1994)]{piterbarg_linear_1994}
Piterbarg  V and Seleznjev  O.
\newblock Linear interpolation of random processes and extremes of a sequence
  of {G}aussian nonstationary processes.
\newblock 1994.
\newblock Technical report 446, Department of Statistics, University of North
  Carolina, Chapel Hill, {NC}.

\bibitem[Rioul(1992)]{rioul_simple_1992}
Rioul  O.
\newblock Simple regularity criteria for subdivision schemes.
\newblock \emph{{SIAM} Journal on Mathematical Analysis}, 23\penalty0
  (6):\penalty0 1544{\textendash}1576, 1992.
\newblock \href{http://dx.doi.org/10.1137/0523086}{doi:10.1137/0523086}

\bibitem[Smirnov(1950)]{smirnov_construction_1950}
Smirnov  N V.
\newblock On the construction of confidence regions for the density of
  distribution of random variables.
\newblock \emph{Doklady Akad. Nauk {SSSR} {(N.S.)}}, 74:\penalty0
  189{\textendash}191, 1950.

\bibitem[Tsybakov(2009)]{tsybakov_introduction_2009}
Tsybakov  A B.
\newblock \emph{Introduction to Nonparametric Estimation}.
\newblock Springer Series in Statistics. Springer, New York, 2009.

\end{thebibliography}

\end{document}